\documentclass[numbers=enddot,12pt,final,onecolumn,notitlepage]{scrartcl}%
\usepackage[headsepline,footsepline,manualmark]{scrlayer-scrpage}
\usepackage{amssymb}
\usepackage{amsmath}
\usepackage{amsthm}
\usepackage{framed}
\usepackage{comment}
\usepackage{color}
\usepackage[breaklinks=True]{hyperref}
\usepackage[sc]{mathpazo}
\usepackage[T1]{fontenc}
\usepackage{needspace}
\usepackage{tabls}
\usepackage{tikz}
\providecommand{\U}[1]{\protect\rule{.1in}{.1in}}
\theoremstyle{definition}
\newtheorem{theo}{Theorem}[section]
\newenvironment{theorem}[1][]
{\begin{theo}[#1]\begin{leftbar}}
{\end{leftbar}\end{theo}}
\newtheorem{lem}[theo]{Lemma}
\newenvironment{lemma}[1][]
{\begin{lem}[#1]\begin{leftbar}}
{\end{leftbar}\end{lem}}
\newtheorem{prop}[theo]{Proposition}
\newenvironment{proposition}[1][]
{\begin{prop}[#1]\begin{leftbar}}
{\end{leftbar}\end{prop}}
\newtheorem{defi}[theo]{Definition}
\newenvironment{definition}[1][]
{\begin{defi}[#1]\begin{leftbar}}
{\end{leftbar}\end{defi}}
\newtheorem{remk}[theo]{Remark}

\newtheorem{coro}[theo]{Corollary}

\newtheorem{conv}[theo]{Convention}
\newenvironment{convention}[1][]
{\begin{conv}[#1]\begin{leftbar}}
{\end{leftbar}\end{conv}}
\newtheorem{quest}[theo]{Question}
\newenvironment{question}[1][]
{\begin{quest}[#1]\begin{leftbar}}
{\end{leftbar}\end{quest}}
\newtheorem{warn}[theo]{Warning}
\newenvironment{warning}[1][]
{\begin{warn}[#1]\begin{leftbar}}
{\end{leftbar}\end{warn}}
\newtheorem{conj}[theo]{Conjecture}

\newtheorem{exam}[theo]{Example}
\newenvironment{example}[1][]
{\begin{exam}[#1]\begin{leftbar}}
{\end{leftbar}\end{exam}}
\newenvironment{statement}{\begin{quote}}{\end{quote}}

\let\sumnonlimits\sum
\let\prodnonlimits\prod
\let\cupnonlimits\bigcup
\let\capnonlimits\bigcap
\renewcommand{\sum}{\sumnonlimits\limits}
\renewcommand{\prod}{\prodnonlimits\limits}
\renewcommand{\bigcup}{\cupnonlimits\limits}
\renewcommand{\bigcap}{\capnonlimits\limits}
\newenvironment{noncompile}{}{}
\excludecomment{verlong}
\includecomment{vershort}
\excludecomment{noncompile}

\usetikzlibrary{arrows.meta}
\usetikzlibrary{calc}
\usetikzlibrary{chains}
\usetikzlibrary{shapes}
\usetikzlibrary{decorations.pathmorphing}
\usetikzlibrary{lindenmayersystems}
\definecolor{darkgreen}{rgb}{0,.5,0}
\newtheoremstyle{plainsl}
{8pt plus 2pt minus 4pt}
{8pt plus 2pt minus 4pt}
{\slshape}
{0pt}
{\bfseries}
{.}
{5pt plus 1pt minus 1pt}
{}
\theoremstyle{plainsl}
\ihead{An equality for balanced digraphs, version 2026-06-01}
\ohead{page \thepage}
\begin{document}

\title{An equality for balanced digraphs}
\author{Darij Grinberg\thanks{Drexel University, Philadelphia, PA.
\href{mailto:darijgrinberg@gmail.com}{\texttt{darijgrinberg@gmail.com}}},
Benjamin Liber\thanks{Drexel University, Philadelphia, PA.
\href{mailto:bl839@drexel.edu}{\texttt{bl839@drexel.edu}}}}
\date{June 1, 2026}
\maketitle

\begin{abstract}
Consider a directed multigraph $D$ that is balanced (i.e., at each vertex, the
indegree equals the outdegree). Let $A$ be its set of arcs. Fix an integer
$k$. Let $s$ be a vertex of $D$. We show that the number of $k$-element
subsets $B$ of $A$ that contain no cycles but contain a path from each vertex
to $s$ (we call them \textquotedblleft$s$-convergences\textquotedblright) is
independent of $s$. This generalizes known facts about spanning arborescences,
acyclic orientations and maximal acyclic subdigraphs (or, equivalently,
minimum feedback arc sets). Moreover, this result can be generalized even
further, replacing \textquotedblleft contain no cycles\textquotedblright\ with
\textquotedblleft have a given set of cycles\textquotedblright.

\end{abstract}

\section{The theorems}

In this note, we shall discuss \emph{balanced multidigraphs} -- i.e., directed
multigraphs (allowing loops and multiple arcs) in which each vertex satisfies
\textquotedblleft outdegree = indegree\textquotedblright. We recall the
relevant definitions in more detail:

A \emph{multidigraph} (henceforth just \emph{digraph}) means a triple
$(V,A,\psi)$, where $V$ and $A$ are two finite sets and $\psi:A\rightarrow
V\times V$ is a map. The elements of $V$ are called the \emph{vertices} of
this digraph, and the elements of $A$ are called the \emph{arcs} of this
digraph. The \emph{source} and \emph{target} of an arc $a\in A$ are,
respectively, the first and second entries of the pair $\psi(a)$. The
\emph{indegree} $\deg^{-}v$ of a vertex $v\in V$ means the number of arcs
$a\in A$ whose target is $v$. The \emph{outdegree} $\deg^{+}v$ of a vertex
$v\in V$ means the number of arcs $a\in A$ whose source is $v$. We say that a
digraph $(V,A,\psi)$ is \emph{balanced} if and only if each vertex $v\in V$
satisfies $\deg^{+}v=\deg^{-}v$. For further terminology on digraphs, we refer
to \cite{22s}.\footnote{The famous directed Euler--Hierholzer theorem (which
will not be used in this note) says that a weakly connected digraph has an
Eulerian circuit if and only if it is balanced. Thus, weakly connected
balanced digraphs are also known as \emph{Eulerian digraphs}.}

A \emph{to-root} of a digraph $D$ means a vertex $s$ of $D$ such that for each
vertex $v$ of $D$, the digraph $D$ has a path from $v$ to $s$ (equivalently, a
walk from $v$ to $s$).

From now on, we \textbf{fix a balanced digraph} $D=(V,A,\psi)$. If $B$ is any
subset of $A$, then $D\left\langle B\right\rangle $ will denote the induced
subdigraph $\left(  V,B,\psi\mid_{B}\right)  $. A subset $B$ of $A$ will be
called \emph{acyclic} if the subdigraph $D\left\langle B\right\rangle $ has no
cycles. (\textquotedblleft Cycle\textquotedblright\ always means
\textquotedblleft directed cycle\textquotedblright, as in \cite[Definition
4.5.1]{22s}.)

Given a vertex $s$ of $D$, we define an \emph{$s$-convergence} to be an
acyclic subset $B$ of $A$ such that $s$ is a to-root of the subdigraph $D
\left\langle B \right\rangle $.

We can think of an $s$-convergence as a set $B$ of arcs of $D$ such that by
following the $B$-arcs (i.e. the arcs in $B$) from any vertex $v\in V$, we
will always arrive at $s$ (no matter which $B$-arcs we take), and we will be
stuck at $s$. (This intuition is formalized in Proposition \ref{prop.s-conv.2}.)

For any $k\in\mathbb{N}$ and $s\in V$, we let $\gamma_{k}\left(  s\right)  $
denote the number of $s$-convergences of size $k$ (that is, with $k$
arcs).\footnote{The symbol $\mathbb{N}$ denotes the set $\left\{
0,1,2,\ldots\right\}  $.}

Our first main result is the following:

\begin{theorem}
\label{thm.balgamma}Let $k\in\mathbb{N}$. Then, the number $\gamma_{k}\left(
s\right)  $ does not depend on $s$. That is, $\gamma_{k}\left(  s\right)
=\gamma_{k}\left(  t\right)  $ for any $s,t\in V$.
\end{theorem}

\begin{example}
\label{exa.1}Let $D$ be the following balanced multidigraph:%
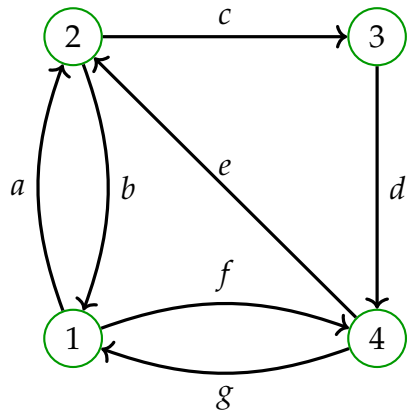
\[%
%TCIMACRO{\TeXButton{tikz multidigraph}{\begin{tikzpicture}[scale=4]
%\begin{scope}[every node/.style={circle,thick,draw=green!60!black}]
%\node(1) at (0,0) {$1$};
%\node(2) at (0,1) {$2$};
%\node(3) at (1,1) {$3$};
%\node(4) at (1,0) {$4$};
%\end{scope}
%\begin{scope}[every edge/.style={draw=black,very thick}, every loop/.style={}]
%\path[->] (1) edge[bend left=20] node[left] {$a$} (2);
%\path[->] (2) edge[bend left=20] node[right] {$b$} (1);
%\path[->] (2) edge node[above] {$c$} (3);
%\path[->] (3) edge node[right] {$d$} (4);
%\path[->] (4) edge node[above] {$e$} (2);
%\path[->] (1) edge[bend left=20] node[above] {$f$} (4);
%\path[->] (4) edge[bend left=20] node[below] {$g$} (1);
%\end{scope}
%\end{tikzpicture}}}%
%BeginExpansion
\begin{tikzpicture}[scale=4]
\begin{scope}[every node/.style={circle,thick,draw=green!60!black}]
\node(1) at (0,0) {$1$};
\node(2) at (0,1) {$2$};
\node(3) at (1,1) {$3$};
\node(4) at (1,0) {$4$};
\end{scope}
\begin{scope}[every edge/.style={draw=black,very thick}, every loop/.style={}]
\path[->] (1) edge[bend left=20] node[left] {$a$} (2);
\path[->] (2) edge[bend left=20] node[right] {$b$} (1);
\path[->] (2) edge node[above] {$c$} (3);
\path[->] (3) edge node[right] {$d$} (4);
\path[->] (4) edge node[above] {$e$} (2);
\path[->] (1) edge[bend left=20] node[above] {$f$} (4);
\path[->] (4) edge[bend left=20] node[below] {$g$} (1);
\end{scope}
\end{tikzpicture}%
%EndExpansion
\]
Then, the $1$-convergences are the subsets%
\[
\left\{  b,d,g\right\}  ,\ \ \ \ \ \ \ \ \ \ \left\{  b,d,e\right\}
,\ \ \ \ \ \ \ \ \ \ \left\{  c,d,g\right\}  ,\ \ \ \ \ \ \ \ \ \ \left\{
b,d,e,g\right\}  ,\ \ \ \ \ \ \ \ \ \ \left\{  b,c,d,g\right\}  .
\]
Hence, $\gamma_{3}\left(  1\right)  =3$, $\gamma_{4}\left(  1\right)  =2$, and
$\gamma_{k}\left(  1\right)  =0$ for all $k\notin\left\{  3,4\right\}  $. As a
visual aid, below is the spanning subdigraph $D\left\langle B\right\rangle $
for $B=\left\{  b,d,e,g\right\}  $:
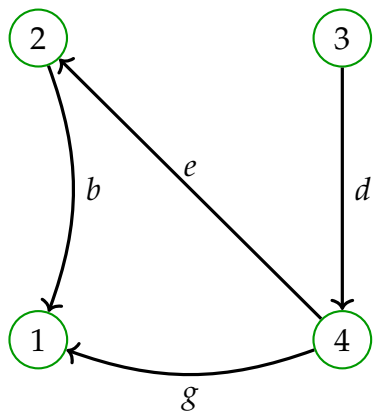
\[%
%TCIMACRO{\TeXButton{tikz multidigraph}{\begin{tikzpicture}[scale=4]
%\begin{scope}[every node/.style={circle,thick,draw=green!60!black}]
%\node(1) at (0,0) {$1$};
%\node(2) at (0,1) {$2$};
%\node(3) at (1,1) {$3$};
%\node(4) at (1,0) {$4$};
%\end{scope}
%\begin{scope}[every edge/.style={draw=black,very thick}, every loop/.style={}]
%\path[->] (2) edge[bend left=20] node[right] {$b$} (1);
%\path[->] (3) edge node[right] {$d$} (4);
%\path[->] (4) edge node[above] {$e$} (2);
%\path[->] (4) edge[bend left=20] node[below] {$g$} (1);
%\end{scope}
%\end{tikzpicture}}}%
%BeginExpansion
\begin{tikzpicture}[scale=4]
\begin{scope}[every node/.style={circle,thick,draw=green!60!black}]
\node(1) at (0,0) {$1$};
\node(2) at (0,1) {$2$};
\node(3) at (1,1) {$3$};
\node(4) at (1,0) {$4$};
\end{scope}
\begin{scope}[every edge/.style={draw=black,very thick}, every loop/.style={}]
\path[->] (2) edge[bend left=20] node[right] {$b$} (1);
\path[->] (3) edge node[right] {$d$} (4);
\path[->] (4) edge node[above] {$e$} (2);
\path[->] (4) edge[bend left=20] node[below] {$g$} (1);
\end{scope}
\end{tikzpicture}%
%EndExpansion
\]

\noindent Theorem \ref{thm.balgamma} says that for all $v\in\{1,2,3,4\}$, we
have $\gamma_{3}(v)=3$, $\gamma_{4}(v)=2$, and $\gamma_{k}(v)=0$ for all
$k\notin\{3,4\}$. For example, the $2$-convergences are the subsets
\[
\left\{  a,d,e\right\}  ,\ \ \ \ \ \ \ \ \ \ \left\{  a,d,g\right\}
,\ \ \ \ \ \ \ \ \ \ \left\{  d,e,f\right\}  ,\ \ \ \ \ \ \ \ \ \ \left\{
a,d,e,f\right\}  ,\ \ \ \ \ \ \ \ \ \ \left\{  a,d,e,g\right\}  .
\]
Hence, $\gamma_{3}(2)=3=\gamma_{3}(1)$, $\gamma_{4}(2)=2=\gamma_{4}(1)$, and
$\gamma_{k}(2)=0=\gamma_{k}(1)$ for all $k\notin\{3,4\}$. The same holds for
$v\in\left\{  3,4\right\}  $.
\end{example}

As we will see in Section \ref{sec.partic}, Theorem \ref{thm.balgamma}
generalizes several results known from the literature. However, Theorem
\ref{thm.balgamma} can be generalized even further. Before we state the most
general version, let us tease a corollary of it that is even easier to state
than Theorem \ref{thm.balgamma}: It says that we can essentially drop the
acyclicity requirement in Theorem \ref{thm.balgamma}!

Let us state this rigorously. Given a vertex $s$ of $D$, we define an
\emph{$s$-preconvergence} to be a (not necessarily acyclic) subset $B$ of $A$
such that $s$ is a to-root of the subdigraph $D\left\langle B\right\rangle $.
(Thus, it is defined just like an $s$-convergence, but without requiring it to
be acyclic.) For any $k\in\mathbb{N}$ and $s\in V$, we let $\delta_{k}\left(
s\right)  $ denote the number of $s$-preconvergences of size $k$ (that is,
with $k$ arcs). Then we claim:

\begin{theorem}
\label{thm.baldelta}Let $k\in\mathbb{N}$. Then, the number $\delta_{k}\left(
s\right)  $ does not depend on $s$. That is, $\delta_{k}\left(  s\right)
=\delta_{k}\left(  t\right)  $ for any $s,t\in V$.
\end{theorem}

\begin{example}
\label{exa.1del}Let $D$ be as in Example \ref{exa.1}. Then, for each $s\in V$,
the numbers $\delta_{k}\left(  s\right)  $ for $k=0,1,2,3,\ldots$ are
$0,0,0,3,10,12,6,1,0,0,0,\ldots$. For instance, $\left\{  c,d,f,g\right\}  $
is a $1$-preconvergence, one among the $10$ counted in $\delta_{4}\left(
1\right)  =10$.
\end{example}

Finally, it is time to let the cat out of the bag and state the most general
version of our result. Rather than removing the acyclicity requirement from
Theorem \ref{thm.balgamma}, we shall replace it by a precise set of cycles
that we want our induced subdigraph $D\left\langle B\right\rangle $ to have.

To state this properly, we introduce some more notations. For any subset $B$
of $A$, we let $\operatorname*{Cycs}B$ denote the set of all cycles of the
subdigraph $D\left\langle B\right\rangle $. (Thus, $B$ is acyclic if and only
if $\operatorname*{Cycs}B=\varnothing$. Note that $\operatorname*{Cycs}A$ is
the set of all cycles of $D\left\langle A\right\rangle =D$.) For any
$k\in\mathbb{N}$ and $s\in V$ and any subset $Z$ of $\operatorname*{Cycs}A$,
we let $\gamma_{k,Z}\left(  s\right)  $ denote the number of $s$%
-preconvergences $B$ of size $k$ (that is, with $k$ arcs) that satisfy
$\operatorname*{Cycs}B=Z$. Then our general result is the following:

\begin{theorem}
\label{thm.gen}Let $k\in\mathbb{N}$. Let $Z$ be a subset of
$\operatorname*{Cycs}A$. Then, the number $\gamma_{k,Z}\left(  s\right)  $
does not depend on $s$. That is, $\gamma_{k,Z}\left(  s\right)  =\gamma
_{k,Z}\left(  t\right)  $ for any $s,t\in V$.
\end{theorem}

\begin{example}
\label{exa.1Z}Let $D$ be as in Example \ref{exa.1}, and let $Z=\left\{
\mathbf{c}\right\}  $, where $\mathbf{c}$ is the cycle $\left(
1,a,2,b,1\right)  $ (that is, the cycle with vertices $1$ and $2$ and arcs $a$
and $b$). Then, the $1$-preconvergences $B$ satisfying $\operatorname*{Cycs}%
B=Z$ are%
\[
\left\{  a,b,d,e\right\}  ,\ \ \ \ \ \ \ \ \ \ \left\{  a,b,d,g\right\}
,\ \ \ \ \ \ \ \ \ \ \left\{  a,b,d,e,g\right\}  .
\]
The $2$-preconvergences $B$ satisfying $\operatorname*{Cycs}B=Z$ are exactly
the same. The $3$-preconvergences $B$ satisfying $\operatorname*{Cycs}B=Z$
are
\[
\left\{  a,b,c,e\right\}  ,\ \ \ \ \ \ \ \ \ \ \left\{  a,b,c,g\right\}
,\ \ \ \ \ \ \ \ \ \ \left\{  a,b,c,e,g\right\}  .
\]
Thus, we find, for each $v\in\left\{  1,2,3\right\}  $, that $\gamma
_{4,Z}\left(  v\right)  =2$ and $\gamma_{5,Z}\left(  v\right)  =1$ and
$\gamma_{k,Z}\left(  v\right)  =0$ for all other $k$. The reader can easily
check that this also holds for $v=4$.
\end{example}

Clearly, an $s$-convergence is the same as an acyclic $s$-preconvergence,
i.e., an $s$-preconvergence $B$ that satisfies $\operatorname*{Cycs}%
B=\varnothing$ (since the condition $\operatorname*{Cycs}B=\varnothing$ says
precisely that $D\left\langle B\right\rangle $ has no cycles, i.e., that $B$
is acyclic). Thus, for any $k\in\mathbb{N}$ and $s\in V$, we have $\gamma
_{k}\left(  s\right)  =\gamma_{k,\varnothing}\left(  s\right)  $. Hence,
Theorem \ref{thm.balgamma} is the particular case of Theorem \ref{thm.gen} for
$Z=\varnothing$.

\section{\label{sec.partic}Particular cases}

Theorem \ref{thm.balgamma} was inspired by a talk of Karla Leipold (NORCOM
2025), which made the first author aware of \cite[Lemma 4.1]{LeiVal24} and
\cite[\S 3.1]{PerPha15}. The talk discussed no enumerative questions, yet was
haunted by a perceptible aura of identities. The present note follows this
aura to its source.

Some particular cases of Theorem \ref{thm.balgamma} are known:

\begin{enumerate}
\item When $k=\left\vert V\right\vert -1$, the $s$-convergences $B$ of size
$k$ (or, more precisely, the respective subdigraphs $D\left\langle
B\right\rangle $ of $D$) are precisely the spanning arborescences of $D$
rooted to $s$ (see \cite[Definition 5.10.1 \textbf{(b)}]{22s} for the
definition of this). Indeed, the condition $\left\vert B\right\vert
=\left\vert V\right\vert -1$, combined with the to-rootness of $s$, forces
$D\left\langle B\right\rangle $ to be an arborescence rooted to $s$ (by
\cite[Theorem 5.10.5]{22s}), and conversely, if $D\left\langle B\right\rangle
$ is a spanning arborescence rooted to $s$, then \cite[Theorem 5.10.5]{22s}
shows that $B$ is acyclic and $\left\vert B\right\vert =\left\vert
V\right\vert -1$. Thus, in the case $k=\left\vert V\right\vert -1$, Theorem
\ref{thm.balgamma} is saying that the number of spanning arborescences of
a balanced digraph $D$ rooted to a vertex $s$ does not depend on $s$.
This is a classical result of van Aardenne-Ehrenfest and de Bruijn
\cite[Theorem 6]{AarBru51} (also proved in \cite[Corollary 5.12.1]{22s}).

Likewise, if $k<\left\vert V\right\vert -1$, then Theorem \ref{thm.balgamma}
is just saying that $0=0$, since a spanning subdigraph $D\left\langle
B\right\rangle $ with fewer than $\left\vert V\right\vert -1$ arcs cannot have
a to-root.

\item If $D=G^{\operatorname*{bidir}}$ for some undirected multigraph
$G=\left(  V,E,\varphi\right)  $ (this means that $D$ is obtained from $G$ by
replacing each edge $e$ with two arcs $e^{\rightarrow}$ and $e^{\leftarrow}$,
going in opposite directions), and if $k=\left\vert E\right\vert =\left\vert
A\right\vert /2$, then the $s$-convergences $B$ of size $k$ are just the
acyclic orientations of $G$ with unique sink $s$ (because the acyclicity
condition forbids $B$ from containing both $e^{\rightarrow}$ and
$e^{\leftarrow}$ for any given edge $e\in E$, but the size condition
$\left\vert B\right\vert =k=\left\vert E\right\vert $ forces $B$ to contain at
least one of these two arcs; see Proposition \ref{prop.s-conv.2} for the
uniqueness of the sink). Thus, in this case, Theorem \ref{thm.balgamma} is
saying that the number of acyclic orientations of a given multigraph $G$ with
unique sink $s$ does not depend on $s$. This is part of a result by Greene and
Zaslavsky \cite[Theorem 7.3]{GreZas83}, proved using hyperplane arrangements,
and has recently been reproved combinatorially by Foissy \cite[Proposition
4.6]{Foissy22}. It is also a consequence of \cite[Proposition 5.3]{NoPoSt02}.

\item Up to reversing the directions of the arcs, \cite[Proposition
3.7]{PerPha15} is Theorem \ref{thm.balgamma} for a specific value of $k$ --
namely, for the maximum possible that makes $\gamma_{k}\left(  s\right)  $
nonzero.\footnote{This maximum $k$ has also been considered in \cite[Theorem
1.11]{KalTot23}. Its independence of $s$ was observed in \cite[paragraph below
Theorem 1.9]{KalTot23} as well.}

We note that when $D$ is weakly connected, then this maximum $k$ is also the
maximum size of an acyclic subset of $A$ (not just of an $s$-convergence).
This is a consequence of \cite[Theorem 3.4]{PerPha15}. It thus follows that
finding this maximum $k$ is equivalent to the \emph{maximum acyclic subdigraph
problem} for Eulerian (= weakly connected balanced) digraphs, which is known
to be NP-hard by \cite[Theorem 3.10]{PerPha15} (see also \cite[\S 3.7.1 and
Lemma 4.4.3]{BanGut18}). In the terminology of algorithmic combinatorics, this
problem is often stated in terms of the complement of the acyclic subset; this
complement is known as a \emph{feedback arc set}. In these terms, our
$\gamma_{k}\left(  s\right)  $ counts the feedback arc sets of size
$\left\vert A\right\vert -k$; this counting problem is \#P-complete
\cite[Theorem 5]{Perrot19}.
\end{enumerate}

\begin{noncompile}
TODO: see if something from
https://en.wikipedia.org/wiki/Feedback\_arc\_set\#Equivalences is relevant to us.

(I'm not sure how relevant this is to us). From
\href{https://link.springer.com/chapter/10.1007/978-3-642-45043-3_26}{this},
we can conclude the following. Let $D=(V,A)$ be a simple balanced digraph with
$|V|=n$ and $|A|=m$. If $\deg^{+}v=\deg^{-}v=1$ for all $v\in V$, then the
size of a maximum acyclic subdigraph is at least $m-\frac{n}{3}$, and this
bound is tight. (Thm 1). If $\deg^{+}v=\deg^{-}v\in\{1,2\}$ for all $v\in V$,
then the size of a maximum acyclic subdigraph is at least $\frac{2m}{3}$, and
this bound is tight (Thm. 2 and Corollary 4).

See also https://doi.org/10.1016/S0166-218X(00)00339-5 .
\end{noncompile}

As we already noticed, Theorem \ref{thm.balgamma} is the particular case of
Theorem \ref{thm.gen} for $Z=\varnothing$. We can easily derive Theorem
\ref{thm.baldelta} from Theorem \ref{thm.gen} as well:

\begin{proof}
[Proof of Theorem \ref{thm.baldelta} using Theorem \ref{thm.gen}.]Let $s,t\in
V$. We must prove that $\delta_{k}\left(  s\right)  =\delta_{k}\left(
t\right)  $.

If $B$ is any $s$-preconvergence, then $\operatorname*{Cycs}B\subseteq
\operatorname*{Cycs}A$ (since $D\left\langle B\right\rangle $ is a subdigraph
of $D=D\left\langle A\right\rangle $, so that every cycle of $D\left\langle
B\right\rangle $ is a cycle of $D\left\langle A\right\rangle $). Thus, by the
sum rule,%
\begin{align*}
&  \left(  \text{number of all }s\text{-preconvergences of size }k\right)  \\
&  =\sum_{Z\subseteq\operatorname*{Cycs}A}\underbrace{\left(  \text{number of
all }s\text{-preconvergences }B\text{ of size }k\text{ that satisfy
}\operatorname*{Cycs}B=Z\right)  }_{\substack{=\gamma_{k,Z}\left(  s\right)
\\\text{(by the definition of }\gamma_{k,Z}\left(  s\right)  \text{)}}}\\
&  =\sum_{Z\subseteq\operatorname*{Cycs}A}\gamma_{k,Z}\left(  s\right)  .
\end{align*}
But $\delta_{k}\left(  s\right)  $ was defined as the number of all
$s$-preconvergences of size $k$. Thus,%
\[
\delta_{k}\left(  s\right)  =\left(  \text{number of all }%
s\text{-preconvergences of size }k\right)  =\sum_{Z\subseteq
\operatorname*{Cycs}A}\gamma_{k,Z}\left(  s\right)  .
\]
Similarly,%
\[
\delta_{k}\left(  t\right)  =\sum_{Z\subseteq\operatorname*{Cycs}A}%
\gamma_{k,Z}\left(  t\right)  .
\]
But the sums on the right hand sides of these two equalities are equal, since
their respective addends agree (by Theorem \ref{thm.gen}). Hence, the left
hand sides must also be equal. In other words, $\delta_{k}\left(  s\right)
=\delta_{k}\left(  t\right)  $. This proves Theorem \ref{thm.baldelta}.
\end{proof}

Theorem \ref{thm.baldelta}, too, can be specialized to a result about
orientations of undirected graphs: If $G$ is an undirected multigraph and $s$
is a vertex of $G$, then the number of orientations of $G$ that have $s$ as a
to-root is independent of $s$. This follows by applying Theorem
\ref{thm.baldelta} to $D=G^{\operatorname*{bidir}}$ and $k=\left\vert
E\right\vert $ (the number of edges of $G$). Note that a similar (but not
quite the same) fact is proved in \cite[Theorem 3.4]{Gioan07}.

\section{The proof}

We will prove Theorem~\ref{thm.gen} through a sequence of lemmas, which are
self-contained and might be of independent interest.

If $P$, $Q$ and $S$ are three sets, then the notation \textquotedblleft%
$S=P\sqcup Q$\textquotedblright\ shall mean that $S=P\cup Q$ and $P\cap
Q=\varnothing$. In other words, it shall mean that $S$ is the union of the two
disjoint sets $P$ and $Q$. Of course, if $S=P\sqcup Q$, then $P$ and $Q$ are
subsets of $S$ and we have $\left\vert S\right\vert =\left\vert P\right\vert
+\left\vert Q\right\vert $.

\begin{noncompile}
For two sets $P$ and $Q$, we shall use the notation $P\sqcup Q$ to denote the
\emph{disjoint union} of $P$ and $Q$, that is, the union $P\cup Q$ if the
intersection $P\cap Q$ is empty. Otherwise, $P\sqcup Q$ does not exist. Of
course, if $P\sqcup Q$ exists, then $\left\vert P\sqcup Q\right\vert
=\left\vert P\right\vert +\left\vert Q\right\vert $.
\end{noncompile}

For any subsets $P$ and $Q$ of $V$, we let $A(P,Q)$ denote the set of arcs in
$A$ whose source belongs to $P$ and whose target belongs to $Q$. The following
fact is a simple but crucial property of balanced digraphs:

\begin{proposition}
\label{prop.A-symmetry}Let $P$ and $Q$ be two subsets of $V$ such that
$V=P\sqcup Q$. Then,
\[
\left\vert A\left(  P,Q\right)  \right\vert =\left\vert A\left(  Q,P\right)
\right\vert .
\]

\end{proposition}

\begin{proof}
This is a known fact (see, e.g., \cite[Exercise 9.1]{22s}). The easiest way to
prove it is as follows: By the definition of an outdegree, we
have\footnote{The symbol \textquotedblleft\#\textquotedblright\ means
\textquotedblleft number\textquotedblright.}%
\begin{align}
&  \sum_{p\in P}\deg^{+}p\nonumber\\
&  =\sum_{p\in P}\left(  \text{\# of arcs }a\in A\text{ with source }p\right)
\nonumber\\
&  =\left(  \text{\# of arcs }a\in A\text{ with source in }P\right)
\nonumber\\
&  =\underbrace{\left(  \text{\# of arcs }a\in A\text{ with source in }P\text{
and target in }P\right)  }_{=\left\vert A\left(  P,P\right)  \right\vert
}\nonumber\\
&  \ \ \ \ \ \ \ \ \ \ +\underbrace{\left(  \text{\# of arcs }a\in A\text{
with source in }P\text{ and target in }Q\right)  }_{=\left\vert A\left(
P,Q\right)  \right\vert }\nonumber\\
&  \ \ \ \ \ \ \ \ \ \ \ \ \ \ \ \ \ \ \ \ \left(
\begin{array}
[c]{c}%
\text{since the target of any arc }a\in A\text{ belongs to either }P\text{ or
}Q\text{,}\\
\text{but not to both (because }V=P\sqcup Q\text{)}%
\end{array}
\right) \nonumber\\
&  =\left\vert A\left(  P,P\right)  \right\vert +\left\vert A\left(
P,Q\right)  \right\vert . \label{pf.prop.A-symmetry.sum.1}%
\end{align}
An analogous argument (but with the roles of sources and targets switched)
shows that%
\begin{equation}
\sum_{p\in P}\deg^{-}p=\left\vert A\left(  P,P\right)  \right\vert +\left\vert
A\left(  Q,P\right)  \right\vert . \label{pf.prop.A-symmetry.sum.2}%
\end{equation}

However, each $p\in P$ satisfies $\deg^{+}p=\deg^{-}p$ (since $D$ is
balanced). Thus, the left hand sides of the equalities
(\ref{pf.prop.A-symmetry.sum.1}) and (\ref{pf.prop.A-symmetry.sum.2}) are
equal. Therefore, their right hand sides are equal as well. In other words, we
have%
\[
\left\vert A\left(  P,P\right)  \right\vert +\left\vert A\left(  P,Q\right)
\right\vert =\left\vert A\left(  P,P\right)  \right\vert +\left\vert A\left(
Q,P\right)  \right\vert .
\]
Subtracting $\left\vert A\left(  P,P\right)  \right\vert $ from this equality,
we obtain $\left\vert A\left(  P,Q\right)  \right\vert =\left\vert A\left(
Q,P\right)  \right\vert $. Thus, Proposition \ref{prop.A-symmetry} is proved.
\end{proof}

\begin{noncompile}
OLD\ PROOF\ OF PROPOSITION \ref{prop.A-symmetry}: We have%
\[
A\left(  P,V\right)  =A\left(  P,P\right)  \sqcup A\left(  P,Q\right)
\]
(since the target of an arc belongs to either $P$ or $Q$ but not to both), so
that%
\[
\left\vert A\left(  P,V\right)  \right\vert =\left\vert A\left(  P,P\right)
\right\vert +\left\vert A\left(  P,Q\right)  \right\vert .
\]
However, we can count the arcs in $A\left(  P,V\right)  $ according to their
sources $p\in P$. This yields%
\begin{equation}
\left\vert A\left(  P,V\right)  \right\vert =\sum_{p\in P}\deg^{+}p,
\label{pf.prop.A-symmetry.3}%
\end{equation}
since each arc in $A\left(  P,V\right)  $ is counted in the outdegree of
exactly one $p\in P$ (namely, of its source). Comparing these two equalities,
we obtain%
\begin{equation}
\sum_{p\in P}\deg^{+}p=\left\vert A\left(  P,P\right)  \right\vert +\left\vert
A\left(  P,Q\right)  \right\vert . \label{pf.prop.A-symmetry.4}%
\end{equation}

On the other hand, $A\left(  V,P\right)  =A\left(  P,P\right)  \sqcup A\left(
Q,P\right)  $ (since the source of an arc belongs to either $P$ or $Q$ but not
to both) and thus%
\[
\left\vert A\left(  V,P\right)  \right\vert =\left\vert A\left(  P,P\right)
\right\vert +\left\vert A\left(  Q,P\right)  \right\vert .
\]
Similarly to (\ref{pf.prop.A-symmetry.3}), we can show that%
\[
\left\vert A\left(  V,P\right)  \right\vert =\sum_{p\in P}\deg^{-}p.
\]
Comparing these two equalities, we obtain%
\begin{equation}
\sum_{p\in P}\deg^{-}p=\left\vert A\left(  P,P\right)  \right\vert +\left\vert
A\left(  Q,P\right)  \right\vert . \label{pf.prop.A-symmetry.5}%
\end{equation}

The left hand sides of the equalities (\ref{pf.prop.A-symmetry.4}) and
(\ref{pf.prop.A-symmetry.5}) are equal (since $D$ is balanced, so that each
$p\in P$ satisfies $\deg^{+}p=\deg^{-}p$). Hence, their right hand sides must
be equal as well. That is,%
\[
\left\vert A\left(  P,P\right)  \right\vert +\left\vert A\left(  P,Q\right)
\right\vert =\left\vert A\left(  P,P\right)  \right\vert +\left\vert A\left(
Q,P\right)  \right\vert .
\]
Subtracting $\left\vert A\left(  P,P\right)  \right\vert $ from this equality,
we obtain $\left\vert A\left(  P,Q\right)  \right\vert =\left\vert A\left(
Q,P\right)  \right\vert $. Thus, Proposition \ref{prop.A-symmetry} is proved.
\end{noncompile}

Now, given a subset $B$ of $A$ and two vertices $v,w\in V$, we say that
\textquotedblleft$v$ \emph{can }$B$\emph{-reach} $w$\textquotedblright\ if the
digraph $D\left\langle B\right\rangle $ has a path from $v$ to $w$ (or,
equivalently, a walk from $v$ to $w$).

Fix two vertices $s,t\in V$ and an integer $k\in\mathbb{N}$. We want to show
that $\gamma_{k,Z}\left(  s\right)  =\gamma_{k,Z}\left(  t\right)  $ for any
subset $Z$ of $\operatorname{Cycs} A$.

For any subset $B$ of $A$, we define the subsets%
\begin{align*}
S\left(  B\right)   &  :=\left\{  v\in V\ \mid\ v\text{ can }B\text{-reach
}s\right\}  \ \ \ \ \ \ \ \ \ \ \text{and}\\
T\left(  B\right)   &  :=\left\{  v\in V\ \mid\ v\text{ can }B\text{-reach
}t\right\}  \ \ \ \ \ \ \ \ \ \ \text{of }V.
\end{align*}
We call them the \emph{attraction basins} of $s$ and $t$ with respect to $B$.
Note that $s\in S\left(  B\right)  $ and $t\in T\left(  B\right)  $ always
hold, since each vertex can $B$-reach itself (by a path of length $0$).

We shall now prove several properties of attraction basins, which will be used
in the proof of Lemma \ref{lem.new4} later on. (The adventurous reader can
treat them as exercises, while the impatient reader might prefer to skip them
and return to them as the need arises.)

\begin{lemma}
\label{lem.reach-trans}Let $B$ be a subset of $A$. Let $v,w\in V$ be two
vertices such that $v$ can $B$-reach $w$. Assume that $w\in S\left(  B\right)
$. Then, $v\in S\left(  B\right)  $.
\end{lemma}

\begin{proof}
We have assumed that $v$ can $B$-reach $w$. In other words, the digraph
$D\left\langle B\right\rangle $ has a path $\mathbf{p}$ from $v$ to $w$.
Consider this path $\mathbf{p}$.

Moreover, $w$ can $B$-reach $s$ (since $w\in S\left(  B\right)  $). In other
words, the digraph $D\left\langle B\right\rangle $ has a path $\mathbf{q}$
from $w$ to $s$. Consider this path $\mathbf{q}$.

Concatenating the path $\mathbf{p}$ (from $v$ to $w$) with the path
$\mathbf{q}$ (from $w$ to $s$), we obtain a walk from $v$ to $s$. Thus, the
digraph $D\left\langle B\right\rangle $ has a walk from $v$ to $s$ (namely,
the one we just obtained). Hence, this digraph $D\left\langle B\right\rangle $
also has a path from $v$ to $s$. In other words, $v$ can $B$-reach $s$. That
is, $v\in S\left(  B\right)  $. This proves Lemma \ref{lem.reach-trans}.
\end{proof}

\begin{lemma}
\label{lem.EAQP}Let $P$ and $Q$ be two subsets of $V$ such that $V=P\sqcup Q$.
Let $E$ be a subset of $A$ such that $S\left(  E\right)  =P$. Then, $E\cap
A\left(  Q,P\right)  =\varnothing$.
\end{lemma}

\begin{proof}
Assume the contrary. Thus, there exists some arc $a\in E\cap A\left(
Q,P\right)  $. Consider this $a$. From $a\in E\cap A\left(  Q,P\right)  $, we
obtain $a\in E$ and $a\in A\left(  Q,P\right)  $.

Hence, the arc $a$ belongs to $A\left(  Q,P\right)  $; thus, it has a source
$q\in Q$ and a target $p\in P$. Consider these $q$ and $p$. Clearly, the
digraph $D\left\langle E\right\rangle $ has a walk from $q$ to $p$ (namely,
the walk that consists of the single arc $a\in E$). In other words, $q$ can
$E$-reach $p$. Moreover, $p\in P=S\left(  E\right)  $ (since $S\left(
E\right)  =P$). Hence, Lemma \ref{lem.reach-trans} (applied to $B=E$ and $v=q$
and $w=p$) shows that $q\in S\left(  E\right)  $. Thus, $q\in S\left(
E\right)  =P$, so that $q\notin Q$ (since $V=P\sqcup Q$ shows that the sets
$P$ and $Q$ are disjoint). But this contradicts $q\in Q$. This contradiction
shows that our assumption was false. Hence, Lemma \ref{lem.EAQP} is proven.
\end{proof}

\begin{lemma}
\label{lem.EAPQ}Let $P$ and $Q$ be two subsets of $V$ such that $V=P\sqcup Q$.
Let $E$ be a subset of $A$ such that $T\left(  E\right)  =Q$. Then, $E\cap
A\left(  P,Q\right)  =\varnothing$.
\end{lemma}

\begin{proof}
This is analogous to Lemma \ref{lem.EAQP}. (Just switch the roles of $s$ and
$t$ as well as the roles of $P$ and $Q$.)
\end{proof}

\begin{lemma}
\label{lem.new0}Let $P$ and $Q$ be two subsets of $V$ such that $V=P\sqcup Q$.
Let $E$ be a subset of $A$. Let $C$ be a subset of $A\left(  P,Q\right)  $.
Then, $S\left(  E\right)  =P$ if and only if $S\left(  E\cup C\right)  =P$.
\end{lemma}

\begin{proof}
First, we observe that if $B_{1}$ and $B_{2}$ are two subsets of $A$
satisfying $B_{1}\subseteq B_{2}$, then any path of $D\left\langle
B_{1}\right\rangle $ is a path of $D\left\langle B_{2}\right\rangle $, and
thus we have $S\left(  B_{1}\right)  \subseteq S\left(  B_{2}\right)  $.
Hence, $S\left(  E\right)  \subseteq S\left(  E\cup C\right)  $ (since
$E\subseteq E\cup C$).

Note furthermore that $P$ is disjoint from $Q$ (since $V=P\sqcup Q$).

We must prove the equivalence $\left(  S\left(  E\right)  =P\right)
\ \Longleftrightarrow\ \left(  S\left(  E\cup C\right)  =P\right)  $. We shall
verify the $\Longrightarrow$ and $\Longleftarrow$ directions separately:
\medskip

$\Longrightarrow:$ Assume that $S\left(  E\right)  =P$. We must show that
$S\left(  E\cup C\right)  =P$.

By assumption, we have $P=S\left(  E\right)  \subseteq S\left(  E\cup
C\right)  $. It remains to prove the converse inclusion.

Let $v\in S\left(  E\cup C\right)  $. Thus, $v$ can $E\cup C$-reach $s$. In
other words, the digraph $D\left\langle E\cup C\right\rangle $ has a path
$\mathbf{p}$ from $v$ to $s$. Consider this path $\mathbf{p}$.

We claim that all arcs of $\mathbf{p}$ belong to $E$. Indeed, let us assume
the contrary. Then, at least one of the arcs of $\mathbf{p}$ does not belong
to $E$. Let $c$ be the \textbf{last} arc of $\mathbf{p}$ that does not belong
to $E$, and let $q$ be the target of this arc $c$.

The arc $c$ belongs to $E\cup C$ (since it is part of the path $\mathbf{p}$,
which is a path of $D\left\langle E\cup C\right\rangle $). Since it does not
belong to $E$, it must thus belong to $C$. Thus, $c\in C\subseteq A\left(
P,Q\right)  $, so that the target of $c$ belongs to $Q$. In other words, $q\in
Q$ (since $q$ is the target of $c$).

However, $c$ is the \textbf{last} arc of $\mathbf{p}$ that does not belong to
$E$. Thus, all the arcs of $\mathbf{p}$ that come after $c$ must belong to
$E$. Therefore, the arcs of $\mathbf{p}$ that come after $c$ form a path of
the digraph $D\left\langle E\right\rangle $. This path starts at $q$ (the
target of $c$) and ends at $s$. Hence, we have shown that $D\left\langle
E\right\rangle $ has a path from $q$ to $s$. In other words, $q$ can $E$-reach
$s$, meaning that $q\in S\left(  E\right)  $. Thus, $q\in S\left(  E\right)
=P$, so that $q\notin Q$ (since $P$ is disjoint from $Q$). This contradicts
$q\in Q$.

This contradiction shows that our assumption was false. Hence, all arcs of
$\mathbf{p}$ belong to $E$. Therefore, $\mathbf{p}$ is a path of
$D\left\langle E\right\rangle $. Hence, $v$ can $E$-reach $s$ (by the path
$\mathbf{p}$). That is, $v\in S\left(  E\right)  =P$.

Since we have proved this for each $v\in S\left(  E\cup C\right)  $, we thus
conclude that $S\left(  E\cup C\right)  \subseteq P$. Combining this with
$P\subseteq S\left(  E\cup C\right)  $, we obtain $S\left(  E\cup C\right)
=P$. This proves the \textquotedblleft$\Longrightarrow$\textquotedblright%
\ direction of Lemma \ref{lem.new0}. \medskip

$\Longleftarrow:$ Assume that $S\left(  E\cup C\right)  =P$. We must show that
$S\left(  E\right)  =P$.

We have $S\left(  E\right)  \subseteq S\left(  E\cup C\right)  =P$ (by
assumption). It remains to prove the converse inclusion.

Let $p\in P$. Then, $p\in P=S\left(  E\cup C\right)  $ (by assumption). Hence,
$p$ can $E\cup C$-reach $s$. That is, the digraph $D\left\langle E\cup
C\right\rangle $ has a path $\mathbf{p}$ from $p$ to $s$. Any vertex $v$ of
this path $\mathbf{p}$ must itself belong to $S\left(  E\cup C\right)  $
(since it can $E\cup C$-reach $s$ by walking along $\mathbf{p}$ from $v$ to
$s$), so it cannot belong to $Q$ (since $S\left(  E\cup C\right)  =P$ is
disjoint from $Q$). Therefore, no arc of this path $\mathbf{p}$ can belong to
$A\left(  P,Q\right)  $ (since this would require its target to belong to
$Q$). Thus, all arcs of this path $\mathbf{p}$ belong to $\left(  E\cup
C\right)  \setminus\underbrace{A\left(  P,Q\right)  }_{\supseteq C}%
\subseteq\left(  E\cup C\right)  \setminus C\subseteq E$. Hence, $\mathbf{p}$
is a path of $D\left\langle E\right\rangle $. Consequently, the vertex $p$ can
$E$-reach $s$ (via the path $\mathbf{p}$). In other words, $p\in S\left(
E\right)  $. Since we have proved this for each $p\in P$, we conclude that
$P\subseteq S\left(  E\right)  $. Therefore, $S\left(  E\right)  =P$ (since
$S\left(  E\right)  \subseteq P$). This proves the \textquotedblleft%
$\Longleftarrow$\textquotedblright\ direction of Lemma \ref{lem.new0}.
\end{proof}

\begin{lemma}
\label{lem.newC}Let $P$ and $Q$ be two subsets of $V$ such that $V=P\sqcup Q$.
Let $E$ be a subset of $A$ such that $S\left(  E\right)  =P$. Let $C$ be a
subset of $A\left(  P,Q\right)  $. Then, $\operatorname*{Cycs}\left(  E\cup
C\right)  =\operatorname*{Cycs}E$.
\end{lemma}

\begin{proof}
First, we observe that if $B_{1}$ and $B_{2}$ are two subsets of $A$
satisfying $B_{1}\subseteq B_{2}$, then any cycle of $D\left\langle
B_{1}\right\rangle $ is a cycle of $D\left\langle B_{2}\right\rangle $, and
thus we have $\operatorname*{Cycs}B_{1}\subseteq\operatorname*{Cycs}B_{2}$
(since $\operatorname*{Cycs}B$ denotes the set of all cycles of $D\left\langle
B\right\rangle $). Applying this to $B_{1}=E$ and $B_{2}=E\cup C$, we obtain
$\operatorname*{Cycs}E\subseteq\operatorname*{Cycs}\left(  E\cup C\right)  $
(since $E\subseteq E\cup C$).

We shall now prove that $\operatorname*{Cycs}\left(  E\cup C\right)
\subseteq\operatorname*{Cycs}E$. Indeed, let $\mathbf{c}\in
\operatorname*{Cycs}\left(  E\cup C\right)  $ be arbitrary. Thus, $\mathbf{c}$
is a cycle of $D\left\langle E\cup C\right\rangle $ (since
$\operatorname*{Cycs}\left(  E\cup C\right)  $ is the set of all cycles of
$D\left\langle E\cup C\right\rangle $). We want to show that $\mathbf{c}$ is a
cycle of $D\left\langle E\right\rangle $ as well.

Assume the contrary. Thus, $\mathbf{c}$ is not a cycle of $D\left\langle
E\right\rangle $. In other words, not all arcs of $\mathbf{c}$ belong to $E$.
But $\mathbf{c}$ is a cycle of $D\left\langle E\cup C\right\rangle $; thus,
each arc of $\mathbf{c}$ belongs to $E\cup C$. In other words, each arc of
$\mathbf{c}$ belongs to $E$ or to $C$. Hence, some arc of $\mathbf{c}$ belongs
to $C$ (because not all arcs of $\mathbf{c}$ belong to $E$). This arc must
therefore belong to $A\left(  P,Q\right)  $ as well (since $C\subseteq
A\left(  P,Q\right)  $); in other words, this arc has source in $P$ and target
in $Q$. Thus, the cycle $\mathbf{c}$ contains both a vertex in $P$ and a
vertex in $Q$ (namely, the source and the target of this arc). Consequently,
it must cross from $Q$ to $P$ at some point (since $V=P\sqcup Q$). In other
words, it contains an arc $a$ with source in $Q$ and target in $P$. In other
words, it contains an arc $a\in A\left(  Q,P\right)  $. Consider this arc $a$.

Lemma \ref{lem.EAQP} yields $E\cap A\left(  Q,P\right)  =\varnothing$. Hence,
$a\notin E$ (because otherwise, we would have $a\in E$, which could be
combined with $a\in A\left(  Q,P\right)  $ to obtain $a\in E\cap A\left(
Q,P\right)  =\varnothing$, which is absurd). Combining this with $a\in E\cup
C$ (since $a$ is an arc of $\mathbf{c}$, which is a cycle of $D\left\langle
E\cup C\right\rangle $), we obtain $a\in\left(  E\cup C\right)  \setminus
E\subseteq C\subseteq A\left(  P,Q\right)  $. Thus, the source of $a$ belongs
to $P$. On the other hand, the source of $a$ belongs to $Q$ (since $a\in
A\left(  Q,P\right)  $). The preceding two sentences contradict each other,
since the sets $P$ and $Q$ are disjoint (because $V=P\sqcup Q$).

This contradiction shows that our assumption was wrong. Hence, we have shown
that $\mathbf{c}$ is a cycle of $D\left\langle E\right\rangle $. In other
words, $\mathbf{c}\in\operatorname*{Cycs}E$.

Forget that we fixed $\mathbf{c}$. We thus have shown that $\mathbf{c}%
\in\operatorname*{Cycs}E$ for each $\mathbf{c}\in\operatorname*{Cycs}\left(
E\cup C\right)  $. In other words, $\operatorname*{Cycs}\left(  E\cup
C\right)  \subseteq\operatorname*{Cycs}E$. Combining this with
$\operatorname*{Cycs}E\subseteq\operatorname*{Cycs}\left(  E\cup C\right)  $,
we obtain $\operatorname*{Cycs}\left(  E\cup C\right)  =\operatorname*{Cycs}%
E$. This proves Lemma \ref{lem.newC}.
\end{proof}

\begin{lemma}
\label{lem.newE}Let $P$ and $Q$ be two subsets of $V$ such that $V=P\sqcup Q$
and $t\in Q$. Let $E$ be a subset of $A$ that is disjoint from $A\left(
P,Q\right)  $. Then, $T\left(  E\right)  \subseteq Q$.
\end{lemma}

\begin{proof}
Let $x\in T\left(  E\right)  $ be any vertex. We must show that $x\in Q$.

Assume the contrary. Then, $x\in P$ (since $V=P\sqcup Q$). But $x$ can
$E$-reach $t$ (since $x\in T\left(  E\right)  $). So the digraph
$D\left\langle E\right\rangle $ has a path $\mathbf{x}$ from $x$ to $t$. This
path $\mathbf{x}$ is a path of $D\left\langle E\right\rangle $; thus, all its
arcs belong to $E$.

The path $\mathbf{x}$ starts at a vertex in $P$ (namely, at $x\in P$) and ends
at a vertex in $Q$ (namely, at $t\in Q$). So this path $\mathbf{x}$ must cross
from $P$ to $Q$ at some point (since $V=P\sqcup Q$). In other words, it must
contain an arc that belongs to $A\left(  P,Q\right)  $. But this is
impossible, since all arcs of $\mathbf{x}$ belong to the set $E$, which is
disjoint from $A\left(  P,Q\right)  $. This contradiction shows that $x\in Q$.
Since we have proved this for each $x\in T\left(  E\right)  $, we thus
conclude that $T\left(  E\right)  \subseteq Q$. This proves Lemma
\ref{lem.newE}.
\end{proof}

\begin{lemma}
\label{lem.newD}Let $P$ and $Q$ be two subsets of $V$ such that $V=P\sqcup Q$
and such that $Q$ is nonempty. Let $B$ be a subset of $A$ such that $S\left(
B\right)  =P$ and $S\left(  B\right)  \cup T\left(  B\right)  =V$. Let
$C=B\cap A\left(  P,Q\right)  $. Then, $T\left(  B\setminus C\right)  =Q$.
\end{lemma}

\begin{proof}
From $S\left(  B\right)  \cup T\left(  B\right)  =V$ and $S\left(  B\right)
=P$, we obtain $\underbrace{\left(  S\left(  B\right)  \cup T\left(  B\right)
\right)  }_{=V}\setminus\underbrace{S\left(  B\right)  }_{=P}=V\setminus P=Q$
(since $V=P\sqcup Q$), so that%
\[
Q=\left(  S\left(  B\right)  \cup T\left(  B\right)  \right)  \setminus
S\left(  B\right)  =T\left(  B\right)  \setminus S\left(  B\right)  .
\]

Next, we claim that $t\in Q$. Indeed, there exists some vertex $v\in Q$ (since
$Q$ is nonempty). Consider this $v$. Then, $v\in Q=T\left(  B\right)
\setminus S\left(  B\right)  $, so that $v\in T\left(  B\right)  $ and
$v\notin S\left(  B\right)  $. From $v\in T\left(  B\right)  $, we see that
$v$ can $B$-reach $t$. If we had $t\in S\left(  B\right)  $, then Lemma
\ref{lem.reach-trans} (applied to $w=t$) would thus show that $v\in S\left(
B\right)  $, which would contradict $v\notin S\left(  B\right)  $. Hence, we
cannot have $t\in S\left(  B\right)  $. Hence, $t\notin S\left(  B\right)
=P$, so that $t\in Q$ (since $V=P\sqcup Q$).

From $C=B\cap A\left(  P,Q\right)  $, we obtain
\begin{equation}
B\setminus C=B\setminus\left(  B\cap A\left(  P,Q\right)  \right)  =B\setminus
A\left(  P,Q\right)  . \label{pf.lem.newD.sm}%
\end{equation}
Hence, the set $B\setminus C$ is disjoint from $A\left(  P,Q\right)  $ (since
$B\setminus A\left(  P,Q\right)  $ clearly is). Thus, Lemma \ref{lem.newE}
(applied to $E=B\setminus C$) yields that $T\left(  B\setminus C\right)
\subseteq Q$ (since $t\in Q$).

Now, we shall show that $Q\subseteq T\left(  B\setminus C\right)  $. Indeed,
let $q\in Q$ be any vertex. Then, $q\in Q=T\left(  B\right)  \setminus
S\left(  B\right)  $. That is, $q\in T\left(  B\right)  $ and $q\notin
S\left(  B\right)  $. In particular, $q$ can $B$-reach $t$ (since $q\in
T\left(  B\right)  $). That is, the digraph $D\left\langle B\right\rangle $
has a path $\mathbf{q}$ from $q$ to $t$.

Let $w$ be any vertex of this path $\mathbf{q}$. Then, $q$ can $B$-reach $w$
as well (by following the path $\mathbf{q}$ from $q$ until $w$). If we had
$w\in S\left(  B\right)  $, then this would entail $q\in S\left(  B\right)  $
(by Lemma \ref{lem.reach-trans}, applied to $v=q$), which would contradict
$q\notin S\left(  B\right)  $. Thus, we cannot have $w\in S\left(  B\right)
$. In other words, we have $w\notin S\left(  B\right)  =P$.

Forget that we fixed $w$. Thus, we have showed that $w\notin P$ for any vertex
$w$ of the path $\mathbf{q}$. In other words, no vertex of the path
$\mathbf{q}$ belongs to $P$. Therefore, no arc of the path $\mathbf{q}$
belongs to $A\left(  P,Q\right)  $ (since the source of such an arc would be a
vertex of $\mathbf{q}$ and belong to $P$). Thus, all arcs of $\mathbf{q}$
belong to $B\setminus A\left(  P,Q\right)  =B\setminus C$ (by
(\ref{pf.lem.newD.sm})). Hence, $\mathbf{q}$ is a path of $D\left\langle
B\setminus C\right\rangle $. This shows that $q$ can $B\setminus C$-reach $t$
(via the path $\mathbf{q}$). In other words, $q\in T\left(  B\setminus
C\right)  $. Since we have proved this for each $q\in Q$, we thus conclude
that $Q\subseteq T\left(  B\setminus C\right)  $.

Combining this with $T\left(  B\setminus C\right)  \subseteq Q$, we obtain
$T\left(  B\setminus C\right)  =Q$. This proves Lemma \ref{lem.newD}.
\end{proof}

Now, we take aim at proving Theorem \ref{thm.gen}. Let us fix a subset $Z$ of
$\operatorname*{Cycs}A$.

For any $i\in\mathbb{Z}$, we let $\mathcal{P}_{i}\left(  A\right)  $ denote
the set of all $i$-element subsets of $A$. Define the subset%
\[
U_{k,Z}:=\left\{  B\in\mathcal{P}_{k}\left(  A\right)  \ \mid\ S\left(
B\right)  \cup T\left(  B\right)  =V\text{ and }\operatorname*{Cycs}%
B=Z\right\}  \ \ \ \ \ \ \ \ \ \ \text{of }\mathcal{P}_{k}\left(  A\right)  .
\]

We observe the following:

\begin{lemma}
\label{lem.new-1}We have%
\[
\left\vert \left\{  B\in U_{k,Z}\ \mid\ S\left(  B\right)  =V\right\}
\right\vert =\gamma_{k,Z}\left(  s\right)  .
\]

\end{lemma}

\begin{proof}
Note that a subset $B$ of $A$ is an $s$-preconvergence

\qquad\qquad if and only if $s$ is a to-root of $D\left\langle B\right\rangle
$,

\qquad i.e.,\quad if and only if each vertex $v\in V$ has a path to $s$ in
$D\left\langle B\right\rangle $,

\qquad i.e.,\quad if and only if each vertex $v\in V$ can $B$-reach $s$,

\qquad i.e.,\quad if and only if $S\left(  B\right)  =V$

\noindent(since $S\left(  B\right)  $ is defined as the set of all vertices
$v\in V$ that can $B$-reach $s$). Thus, an $s$-preconvergence is the same
thing as a subset $B$ of $A$ that satisfies $S\left(  B\right)  =V$.

Hence, the number $\gamma_{k,Z}\left(  s\right)  $ of all $s$-preconvergences
$B$ of size $k$ that satisfy $\operatorname*{Cycs}B=Z$ can also be described
as the number of all subsets $B$ of $A$ of size $k$ that satisfy $S\left(
B\right)  =V$ and $\operatorname*{Cycs}B=Z$. In other words,%
\begin{equation}
\gamma_{k,Z}\left(  s\right)  =\left\vert \left\{  B\in\mathcal{P}_{k}\left(
A\right)  \text{ }\mid\ S\left(  B\right)  =V\text{ and }\operatorname*{Cycs}%
B=Z\right\}  \right\vert . \label{pf.lem.new-1.2}%
\end{equation}
But any $B\in\mathcal{P}_{k}\left(  A\right)  $ satisfying $S\left(  B\right)
=V$ and $\operatorname*{Cycs}B=Z$ must also satisfy $B\in U_{k,Z}$ (since
$\underbrace{S\left(  B\right)  }_{=V}\cup\,T\left(  B\right)  =V\cup T\left(
B\right)  =V$). Conversely, the definition of $U_{k,Z}$ shows that any $B\in
U_{k,Z}$ belongs to $\mathcal{P}_{k}\left(  A\right)  $ and satisfies
$\operatorname*{Cycs}B=Z$. These two facts show that%
\[
\left\{  B\in\mathcal{P}_{k}\left(  A\right)  \text{ }\mid\ S\left(  B\right)
=V\text{ and }\operatorname*{Cycs}B=Z\right\}  =\left\{  B\in U_{k,Z}%
\ \mid\ S\left(  B\right)  =V\right\}  .
\]
Hence, we can rewrite (\ref{pf.lem.new-1.2}) as
\[
\gamma_{k,Z}\left(  s\right)  =\left\vert \left\{  B\in U_{k,Z}\ \mid
\ S\left(  B\right)  =V\right\}  \right\vert .
\]
This proves Lemma \ref{lem.new-1}.
\end{proof}

Now we claim the following:

\begin{lemma}
\label{lem.new1} We have%
\[
\gamma_{k,Z}\left(  s\right)  =\left\vert U_{k,Z}\right\vert -\sum
_{\substack{P,Q\subseteq V\text{ are nonempty;}\\V=P\sqcup Q}}\left\vert
\left\{  B\in U_{k,Z}\ \mid\ S\left(  B\right)  =P\right\}  \right\vert .
\]

\end{lemma}

\begin{proof}
First, we note that the set $V$ is nonempty (since $s\in V$).

For each $B\in U_{k,Z}$, there is a unique pair $\left(  P,Q\right)  $ of
subsets $P,Q\subseteq V$ satisfying $V=P\sqcup Q$ such that $S\left(
B\right)  =P$ (indeed, the condition $S\left(  B\right)  =P$ uniquely
determines $P$, while the condition $V=P\sqcup Q$ forces $Q$ to be $V\setminus
P$). Moreover, in this unique pair $\left(  P,Q\right)  $, the set $P$ is
nonempty (since $s\in S\left(  B\right)  =P$). Hence, for each $B\in U_{k,Z}$,
there is a unique pair $\left(  P,Q\right)  $ of subsets $P,Q\subseteq V$
satisfying \textquotedblleft$P$ is nonempty\textquotedblright\ and $V=P\sqcup
Q$ such that $S\left(  B\right)  =P$. Therefore, we can count the elements
$B\in U_{k,Z}$ by the sum rule, according to the value of this pair:%
\begin{align*}
\left\vert U_{k,Z}\right\vert  &  =\sum_{\substack{P,Q\subseteq V;\\P\text{ is
nonempty;}\\V=P\sqcup Q}}\left\vert \left\{  B\in U_{k,Z}\ \mid\ S\left(
B\right)  =P\right\}  \right\vert \\
&  =\sum_{\substack{P,Q\subseteq V\text{ are nonempty;}\\V=P\sqcup
Q}}\left\vert \left\{  B\in U_{k,Z}\ \mid\ S\left(  B\right)  =P\right\}
\right\vert +\underbrace{\left\vert \left\{  B\in U_{k,Z}\ \mid\ S\left(
B\right)  =V\right\}  \right\vert }_{\substack{=\gamma_{k,Z}\left(  s\right)
\\\text{(by Lemma \ref{lem.new-1})}}}\\
&  \ \ \ \ \ \ \ \ \ \ \ \ \ \ \ \ \ \ \ \ \left(
\begin{array}
[c]{c}%
\text{here, we have split off the addend}\\
\text{for }\left(  P,Q\right)  =\left(  V,\varnothing\right)  \text{ from the
sum}\\
\text{(this addend indeed exists because }V\text{ is nonempty)}%
\end{array}
\right) \\
&  =\sum_{\substack{P,Q\subseteq V\text{ are nonempty;}\\V=P\sqcup
Q}}\left\vert \left\{  B\in U_{k,Z}\ \mid\ S\left(  B\right)  =P\right\}
\right\vert +\gamma_{k,Z}\left(  s\right)  .
\end{align*}
Solving this for $\gamma_{k,Z}\left(  s\right)  $, we obtain the claim of the lemma.
\end{proof}

Similarly, we find:

\begin{lemma}
\label{lem.new2}We have%
\[
\gamma_{k,Z}\left(  t\right)  =\left\vert U_{k,Z}\right\vert -\sum
_{\substack{P,Q\subseteq V\text{ are nonempty;}\\V=P\sqcup Q}}\left\vert
\left\{  B\in U_{k,Z}\ \mid\ T\left(  B\right)  =Q\right\}  \right\vert .
\]

\end{lemma}

\begin{proof}
Analogous to the proof of Lemma \ref{lem.new1}. (Just switch the roles of $s$
and $t$ and also the roles of $P$ and $Q$.)
\end{proof}

Our goal is to prove $\gamma_{k,Z}\left(  s\right)  =\gamma_{k,Z}\left(
t\right)  $. In light of Lemma \ref{lem.new1} and Lemma \ref{lem.new2}, it
will suffice to show the following:

\begin{proposition}
\label{prop.new3}Let $P$ and $Q$ be two nonempty subsets of $V$ such that
$V=P\sqcup Q$. Then,%
\[
\left\vert \left\{  B\in U_{k,Z}\ \mid\ S\left(  B\right)  =P\right\}
\right\vert =\left\vert \left\{  B\in U_{k,Z}\ \mid\ T\left(  B\right)
=Q\right\}  \right\vert .
\]

\end{proposition}

We will prove this by finding analogous expressions for both sides. First, we
need some notation:

\begin{definition}
If $P$ and $Q$ are two subsets of $V$ satisfying $V=P\sqcup Q$, and if
$i\in\mathbb{Z}$ is arbitrary, then we set%
\[
X_{i,Z}^{P,Q}:=\left\{  B\in\mathcal{P}_{i}(A)\ \mid\ S(B)=P\text{ and
}T(B)=Q\text{ and }\operatorname*{Cycs}B=Z\right\}  .
\]

\end{definition}

\begin{lemma}
\label{lem.new4}Let $P$ and $Q$ be two nonempty subsets of $V$ such that
$V=P\sqcup Q$. Then,%
\[
\left\vert \left\{  B\in U_{k,Z}\ \mid\ S\left(  B\right)  =P\right\}
\right\vert =\sum_{m\in\mathbb{N}}\dbinom{\left\vert A\left(  P,Q\right)
\right\vert }{m}\cdot\left\vert X_{k-m,Z}^{P,Q}\right\vert .
\]

\end{lemma}

We note that the infinite sum on the right hand side here is well-defined,
since it has only finitely many nonzero addends.\footnote{Indeed, all negative
integers $i$ satisfy $\mathcal{P}_{i}\left(  A\right)  =\varnothing$ and thus
$X_{i,Z}^{P,Q}=\varnothing$, so that $\left\vert X_{i,Z}^{P,Q}\right\vert =0$;
thus, we conclude that $\left\vert X_{k-m,Z}^{P,Q}\right\vert =0$ whenever
$m>k$. Alternatively, we can observe that $\dbinom{\left\vert A\left(
P,Q\right)  \right\vert }{m}=0$ whenever $m>\left\vert A\left(  P,Q\right)
\right\vert $.}

\begin{proof}
[Proof of Lemma \ref{lem.new4}.]By the sum rule,
\begin{align}
&  \left\vert \left\{  B\in U_{k,Z}\ \mid\ S\left(  B\right)  =P\right\}
\right\vert \nonumber\\
&  =\sum_{C\subseteq A\left(  P,Q\right)  }\left\vert \left\{  B\in
U_{k,Z}\ \mid\ S\left(  B\right)  =P\text{ and }B\cap A\left(  P,Q\right)
=C\right\}  \right\vert , \label{pf.lem.new4.sum}%
\end{align}
since the intersection $B\cap A\left(  P,Q\right)  $ is always a subset of
$A\left(  P,Q\right)  $.

Fix a subset $C\subseteq A(P,Q)$, and let
\[
Y_{k,Z}:=\left\{  B\in U_{k,Z}\ \mid\ S\left(  B\right)  =P\text{ and }B\cap
A\left(  P,Q\right)  =C\right\}  .
\]
We want to show that $\left\vert Y_{k,Z}\right\vert =\left\vert
X_{k-\left\vert C\right\vert ,Z}^{P,Q}\right\vert $. Define
\begin{align*}
\text{the map }\Phi:Y_{k,Z}  &  \rightarrow X_{k-\left\vert C\right\vert
,Z}^{P,Q}\\
\text{by }\Phi(B)  &  =B\setminus C
\end{align*}
and
\begin{align*}
\text{the map }\Psi:X_{k-\left\vert C\right\vert ,Z}^{P,Q}  &  \rightarrow
Y_{k,Z}\\
\text{by }\Psi(E)  &  =E\cup C.
\end{align*}
Let us first show that these maps are well-defined:

\begin{statement}
\textit{Claim 1.} The map $\Phi$ is well-defined. That is, $B\setminus C\in
X_{k-\left\vert C\right\vert ,Z}^{P,Q}$ for all $B\in Y_{k,Z}$.
\end{statement}

\begin{proof}
[Proof of Claim 1.]Let $B\in Y_{k,Z}$. Then $B\in U_{k,Z}$ as well as
$S\left(  B\right)  =P$ and $B\cap A\left(  P,Q\right)  =C$ (by the definition
of $Y_{k,Z}$). From $B\in U_{k,Z}$, we see that $B$ is a $k$-element subset of
$A$ satisfying $S\left(  B\right)  \cup T\left(  B\right)  =V$ and
$\operatorname*{Cycs}B=Z$ (by the definition of $U_{k,Z}$). In particular,
$\left\vert B\right\vert =k$.

We must show that $B\setminus C\in X_{k-\left\vert C\right\vert ,Z}^{P,Q}$. In
other words, we must show that $B\setminus C$ is a set in $\mathcal{P}%
_{k-\left\vert C\right\vert }\left(  A\right)  $ and satisfies $S\left(
B\setminus C\right)  =P$ and $T\left(  B\setminus C\right)  =Q$ and
$\operatorname*{Cycs}\left(  B\setminus C\right)  =Z$.

From $C=B\cap A\left(  P,Q\right)  \subseteq B$, we obtain
\[
\left\vert B\setminus C\right\vert =\left\vert B\right\vert -\left\vert
C\right\vert =k-\left\vert C\right\vert \ \ \ \ \ \ \ \ \ \ \left(
\text{since }\left\vert B\right\vert =k\right)  .
\]
Thus, $B\setminus C\in\mathcal{P}_{k-\left\vert C\right\vert }\left(
A\right)  $. Also, from $C\subseteq B$, we obtain $\left(  B\setminus
C\right)  \cup C=B$.

\begin{noncompile}
Note also that%
\[
B\setminus A\left(  P,Q\right)  =B\setminus\underbrace{\left(  B\cap
A(P,Q)\right)  }_{=C}=B\setminus C.
\]
Thus, the set $B\setminus C$ is disjoint from $A\left(  P,Q\right)  $ (since
$B\setminus A\left(  P,Q\right)  $ clearly is).

Next we recall that $S\left(  B\right)  =P$. Hence, Lemma \ref{lem.EAQP}
(applied to $E=B$) shows that $B\cap A\left(  Q,P\right)  =\varnothing$. In
other words, no arc in $B$ belongs to $A\left(  Q,P\right)  $. Hence, no arc
in $B\setminus C$ belongs to $A\left(  Q,P\right)  $ either (since $B\setminus
C\subseteq B$). Furthermore, no arc in $B\setminus C$ belongs to $A\left(
P,Q\right)  $ (since the set $B\setminus C$ is disjoint from $A\left(
P,Q\right)  $).
\end{noncompile}

Let us now prove that $S\left(  B\setminus C\right)  =P$. Indeed, Lemma
\ref{lem.new0} (applied to $E=B\setminus C$) yields that $S\left(  B\setminus
C\right)  =P$ if and only if $S\left(  \left(  B\setminus C\right)  \cup
C\right)  =P$. Since $S\left(  \left(  B\setminus C\right)  \cup C\right)  =P$
does hold (because $S\left(  \underbrace{\left(  B\setminus C\right)  \cup
C}_{=B}\right)  =S\left(  B\right)  =P$), we thus conclude that $S\left(
B\setminus C\right)  =P$.

Hence, Lemma \ref{lem.newC} (applied to $E=B\setminus C$) shows that
$\operatorname*{Cycs}\left(  \left(  B\setminus C\right)  \cup C\right)
=\operatorname*{Cycs}\left(  B\setminus C\right)  $ (since $C\subseteq
A\left(  P,Q\right)  $). Thus, $\operatorname*{Cycs}\left(  B\setminus
C\right)  =\operatorname*{Cycs}\underbrace{\left(  \left(  B\setminus
C\right)  \cup C\right)  }_{=B}=\operatorname*{Cycs}B=Z$.

Finally, Lemma \ref{lem.newD} shows that $T\left(  B\setminus C\right)  =Q$
(since $C=B\cap A\left(  P,Q\right)  $).

Altogether, we now have shown that $B\setminus C$ is a set in $\mathcal{P}%
_{k-\left\vert C\right\vert }\left(  A\right)  $ and satisfies $S\left(
B\setminus C\right)  =P$ and $T\left(  B\setminus C\right)  =Q$ and
$\operatorname*{Cycs}\left(  B\setminus C\right)  =Z$. In other words,
$B\setminus C\in X_{k-\left\vert C\right\vert ,Z}^{P,Q}$ (by the definition of
$X_{k-\left\vert C\right\vert ,Z}^{P,Q}$). Claim 1 is thus proved.
\end{proof}

\begin{statement}
\textit{Claim 2.} The map $\Psi$ is well-defined. That is, $E\cup C\in
Y_{k,Z}$ for all $E\in X_{k-\left\vert C\right\vert ,Z}^{P,Q}$.
\end{statement}

\begin{proof}
[Proof of Claim 2.]Let $E\in X_{k-\left\vert C\right\vert ,Z}^{P,Q}$. Then $E$
is a $\left(  k-\left\vert C\right\vert \right)  $-element subset of $A$
satisfying $S\left(  E\right)  =P$ and $T\left(  E\right)  =Q$ and
$\operatorname*{Cycs}E=Z$ (by the definition of $X_{k-\left\vert C\right\vert
,Z}^{P,Q}$). In particular, $\left\vert E\right\vert =k-\left\vert
C\right\vert $.

We must show that $E\cup C\in Y_{k,Z}$. In other words, we must show that
$E\cup C\in U_{k,Z}$, $S\left(  E\cup C\right)  =P$, and $\left(  E\cup
C\right)  \cap A\left(  P,Q\right)  =C$.

From $S\left(  E\right)  =P$, we immediately obtain $S\left(  E\cup C\right)
=P$ by Lemma \ref{lem.new0}.

Moreover, Lemma \ref{lem.EAPQ} yields $E\cap A\left(  P,Q\right)
=\varnothing$. Also, Lemma \ref{lem.EAQP} shows that $E\cap A\left(
Q,P\right)  =\varnothing$. Furthermore, Lemma \ref{lem.newC} yields
$\operatorname*{Cycs}\left(  E\cup C\right)  =\operatorname*{Cycs}E=Z$.

The set $A\left(  P,Q\right)  $ is disjoint from $A\left(  Q,P\right)  $,
since the source of an arc cannot belong to $P$ and $Q$ at the same time
(because of $V=P\sqcup Q$). Hence, the set $C$ (being a subset of $A\left(
P,Q\right)  $) must be disjoint from $A\left(  Q,P\right)  $ as well. In other
words, $C\cap A\left(  Q,P\right)  =\varnothing$.

Also, note that $C$ is a subset of $A\left(  P,Q\right)  $, and thus we have
\[
E\cap\underbrace{C}_{\subseteq A\left(  P,Q\right)  }\subseteq E\cap A\left(
P,Q\right)  =\varnothing.
\]
Hence,
\begin{equation}
E\cap C=\varnothing, \label{pf.lem.new4.c2.pf.empty}%
\end{equation}
so that%
\[
\left\vert E\cup C\right\vert =\left\vert E\right\vert +\left\vert
C\right\vert =k\ \ \ \ \ \ \ \ \ \ \left(  \text{since }\left\vert
E\right\vert =k-\left\vert C\right\vert \right)  .
\]
Therefore, $E\cup C\in\mathcal{P}_{k}\left(  A\right)  $.

Furthermore,%
\[
\left(  E\cup C\right)  \cap A\left(  P,Q\right)  =\underbrace{\left(  E\cap
A\left(  P,Q\right)  \right)  }_{=\varnothing}\cup\underbrace{\left(  C\cap
A\left(  P,Q\right)  \right)  }_{\substack{=C\\\text{(since }C\subseteq
A\left(  P,Q\right)  \text{)}}}=\varnothing\cup C=C.
\]

\begin{noncompile}
Also,%
\[
\left(  E\cup C\right)  \cap A\left(  Q,P\right)  =\underbrace{\left(  E\cap
A\left(  Q,P\right)  \right)  }_{=\varnothing}\cup\underbrace{\left(  C\cap
A\left(  Q,P\right)  \right)  }_{=\varnothing}=\varnothing\cup\varnothing
=\varnothing.
\]

\end{noncompile}

If $v$ is a vertex in $S\left(  E\right)  $, then $v$ can $E$-reach $s$ and
thus can $E\cup C$-reach $s$ as well (since $E\subseteq E\cup C$); thus it
belongs to $S\left(  E\cup C\right)  $. This shows that $S\left(  E\right)
\subseteq S\left(  E\cup C\right)  $. Similarly, $T\left(  E\right)  \subseteq
T\left(  E\cup C\right)  $. Now,
\[
V=P\sqcup Q=\underbrace{P}_{=S\left(  E\right)  \subseteq S\left(  E\cup
C\right)  }\cup\underbrace{Q}_{=T\left(  E\right)  \subseteq T\left(  E\cup
C\right)  }\subseteq S\left(  E\cup C\right)  \cup T\left(  E\cup C\right)  .
\]
Of course, this entails $S\left(  E\cup C\right)  \cup T\left(  E\cup
C\right)  =V$.

\begin{noncompile}
Each vertex in $P$ can $E\cup C$-reach $s$ (since it lies in $P=S\left(  E\cup
C\right)  $). Each vertex in $Q$ can $E$-reach $t$ (since it lies in
$Q=T\left(  E\right)  $) and thus can $E\cup C$-reach $t$ as well (since
$E\subseteq E\cup C$). Since each vertex in $V$ belongs to either $P$ or $Q$
(because $V=P\sqcup Q$), we thus conclude that each vertex in $V$ can $E\cup
C$-reach $s$ (if it belongs to $P$) or can $E\cup C$-reach $t$ (if it belongs
to $Q$). In other words, each vertex in $V$ belongs to $S\left(  E\cup
C\right)  $ or to $T\left(  E\cup C\right)  $. In other words, $V\subseteq
S\left(  E\cup C\right)  \cup T\left(  E\cup C\right)  $. Of course, this
entails $S\left(  E\cup C\right)  \cup T\left(  E\cup C\right)  =V$.
\end{noncompile}

This shows that $E\cup C\in U_{k,Z}$ (since $E\cup C\in\mathcal{P}_{k}\left(
A\right)  $ and $\operatorname*{Cycs}\left(  E\cup C\right)  =Z$).

So, we have shown that $E\cup C\in U_{k,Z}$, and $S\left(  E\cup C\right)
=P$, and $\left(  E\cup C\right)  \cap A\left(  P,Q\right)  =C$. All together,
this yields $E\cup C\in Y_{k,Z}$. Claim 2 is thus proved.
\end{proof}

\begin{statement}
\textit{Claim 3.} The maps $\Phi$ and $\Psi$ are mutually inverse.
\end{statement}

\begin{proof}
[Proof of Claim 3.]For each $B\in Y_{k,Z}$, we have $B\cap A\left(
P,Q\right)  =C$ and therefore $C=B\cap A\left(  P,Q\right)  \subseteq B$.
Thus, for each $B\in Y_{k,Z}$, we have%
\[
\Psi\left(  \Phi\left(  B\right)  \right)  =\Psi\left(  B\setminus C\right)
=\left(  B\setminus C\right)  \cup C=B\ \ \ \ \ \ \ \ \ \ \left(  \text{since
}C\subseteq B\right)  .
\]
In other words, $\Psi\circ\Phi=\operatorname*{id}$.

On the other hand, for each $E\in X_{k-\left\vert C\right\vert }^{P,Q}$, we
have $E\cap C=\varnothing$ (see the equality (\ref{pf.lem.new4.c2.pf.empty})
in the proof of Claim 2) and thus%
\[
\Phi\left(  \Psi\left(  E\right)  \right)  =\Phi\left(  E\cup C\right)
=\left(  E\cup C\right)  \setminus C=E\ \ \ \ \ \ \ \ \ \ \left(  \text{since
}E\cap C=\varnothing\right)  .
\]
This shows that $\Phi\circ\Psi=\operatorname*{id}$.

From $\Phi\circ\Psi=\operatorname*{id}$ and $\Psi\circ\Phi=\operatorname*{id}%
$, we conclude that the maps $\Phi$ and $\Psi$ are mutually inverse. This
proves Claim 3.
\end{proof}

Claim 3 shows that the map $\Psi$ is a bijection from $X_{k-\left\vert
C\right\vert ,Z}^{P,Q}$ to $Y_{k,Z}$. Thus,%
\[
\left\vert X_{k-\left\vert C\right\vert ,Z}^{P,Q}\right\vert =\left\vert
Y_{k,Z}\right\vert =\left\vert \left\{  B\in U_{k,Z}\ \mid\ S\left(  B\right)
=P\text{ and }B\cap A\left(  P,Q\right)  =C\right\}  \right\vert
\]
(by the definition of $Y_{k,Z}$).

We have proved this equality for each subset $C$ of $A\left(  P,Q\right)  $.
Summing it over all such subsets, we obtain%
\begin{align*}
\sum_{C\subseteq A\left(  P,Q\right)  }\left\vert X_{k-\left\vert C\right\vert
,Z}^{P,Q}\right\vert  &  =\sum_{C\subseteq A\left(  P,Q\right)  }\left\vert
\left\{  B\in U_{k,Z}\ \mid\ S\left(  B\right)  =P\text{ and }B\cap A\left(
P,Q\right)  =C\right\}  \right\vert \\
&  =\left\vert \left\{  B\in U_{k,Z}\ \mid\ S\left(  B\right)  =P\right\}
\right\vert \ \ \ \ \ \ \ \ \ \ \left(  \text{by (\ref{pf.lem.new4.sum}%
)}\right)  .
\end{align*}
Thus,%
\begin{align*}
\left\vert \left\{  B\in U_{k,Z}\ \mid\ S\left(  B\right)  =P\right\}
\right\vert  &  =\sum_{C\subseteq A\left(  P,Q\right)  }\left\vert
X_{k-\left\vert C\right\vert ,Z}^{P,Q}\right\vert =\sum_{m\in\mathbb{N}%
}\ \ \sum_{\substack{C\subseteq A\left(  P,Q\right)  ;\\\left\vert
C\right\vert =m}}\underbrace{\left\vert X_{k-\left\vert C\right\vert ,Z}%
^{P,Q}\right\vert }_{\substack{=\left\vert X_{k-m,Z}^{P,Q}\right\vert
\\\text{(since }\left\vert C\right\vert =m\text{)}}}\\
&  =\sum_{m\in\mathbb{N}}\underbrace{\sum_{\substack{C\subseteq A\left(
P,Q\right)  ;\\\left\vert C\right\vert =m}}\left\vert X_{k-m,Z}^{P,Q}%
\right\vert }_{\substack{=\dbinom{\left\vert A\left(  P,Q\right)  \right\vert
}{m}\cdot\left\vert X_{k-m,Z}^{P,Q}\right\vert \\\text{(since this is a sum of
}\dbinom{\left\vert A\left(  P,Q\right)  \right\vert }{m}\\\text{many equal
addends)}}}\\
&  =\sum_{m\in\mathbb{N}}\dbinom{\left\vert A\left(  P,Q\right)  \right\vert
}{m}\cdot\left\vert X_{k-m,Z}^{P,Q}\right\vert .
\end{align*}
This proves Lemma \ref{lem.new4}.
\end{proof}

\begin{lemma}
\label{lem.new5} Let $P$ and $Q$ be two nonempty subsets of $V$ such that
$V=P\sqcup Q$. Then,%
\[
\left\vert \left\{  B\in U_{k,Z}\ \mid\ T\left(  B\right)  =Q\right\}
\right\vert =\sum_{m\in\mathbb{N}}\dbinom{\left\vert A\left(  Q,P\right)
\right\vert }{m}\cdot\left\vert X_{k-m,Z}^{P,Q}\right\vert .
\]

\end{lemma}

\begin{proof}
Analogous to the proof of Lemma \ref{lem.new4}. (Just switch the roles of $s$
and $t$ and also the roles of $P$ and $Q$.)
\end{proof}

\begin{proof}
[Proof of Proposition \ref{prop.new3}.]Lemma \ref{lem.new4} yields%
\begin{align*}
\left\vert \left\{  B\in U_{k,Z}\ \mid\ S\left(  B\right)  =P\right\}
\right\vert  &  =\sum_{m\in\mathbb{N}}\dbinom{\left\vert A\left(  P,Q\right)
\right\vert }{m}\cdot\left\vert X_{k-m,Z}^{P,Q}\right\vert \\
&  =\sum_{m\in\mathbb{N}}\dbinom{\left\vert A\left(  Q,P\right)  \right\vert
}{m}\cdot\left\vert X_{k-m,Z}^{P,Q}\right\vert \\
&  \ \ \ \ \ \ \ \ \ \ \ \ \ \ \ \ \ \ \ \ \left(
\begin{array}
[c]{c}%
\text{since Proposition \ref{prop.A-symmetry}}\\
\text{yields }\left\vert A\left(  P,Q\right)  \right\vert =\left\vert A\left(
Q,P\right)  \right\vert
\end{array}
\right) \\
&  =\left\vert \left\{  B\in U_{k,Z}\ \mid\ T\left(  B\right)  =Q\right\}
\right\vert \ \ \ \ \ \ \ \ \ \ \left(  \text{by Lemma \ref{lem.new5}}\right)
.
\end{align*}
This proves Proposition \ref{prop.new3}.
\end{proof}

Now, Lemma \ref{lem.new1} yields%
\begin{align*}
\gamma_{k,Z}\left(  s\right)   &  =\left\vert U_{k,Z}\right\vert
-\sum_{\substack{P,Q\subseteq V\text{ are nonempty;}\\V=P\sqcup Q}%
}\underbrace{\left\vert \left\{  B\in U_{k,Z}\ \mid\ S\left(  B\right)
=P\right\}  \right\vert }_{\substack{=\left\vert \left\{  B\in U_{k,Z}%
\ \mid\ T\left(  B\right)  =Q\right\}  \right\vert \\\text{(by Proposition
\ref{prop.new3})}}}\\
&  =\left\vert U_{k,Z}\right\vert -\sum_{\substack{P,Q\subseteq V\text{ are
nonempty;}\\V=P\sqcup Q}}\left\vert \left\{  B\in U_{k,Z}\ \mid\ T\left(
B\right)  =Q\right\}  \right\vert =\gamma_{k,Z}\left(  t\right)
\end{align*}
(by Lemma \ref{lem.new2}). This completes the proof of Theorem \ref{thm.gen}.

As we already said, Theorem \ref{thm.balgamma} and Theorem \ref{thm.baldelta}
follow from Theorem \ref{thm.gen}.

\section{Further remarks}

\textbf{\quad1.} One might wonder whether our proof of Theorem \ref{thm.gen}
is, or can be made, bijective. In the given form, it is not, as it uses
subtraction twice: once in proving Proposition \ref{prop.A-symmetry} and once
again in the \textquotedblleft subtractive flip\textquotedblright\ that is
involved in Lemma \ref{lem.new1} (and Lemma \ref{lem.new2}). Both of these
instances of subtraction can be made bijective using the Garsia--Milne
involution principle \cite[\S 4.6]{StaWhi86}.

In the case of Lemma \ref{lem.new1}, only the simplest case of the involution
principle is needed: Let $\Gamma_{k,Z}\left(  s\right)  $ be the set of all
$s$-preconvergences $B$ of size $k$ that satisfy $\operatorname*{Cycs}B=Z$,
and let $\Gamma_{k,Z}\left(  t\right)  $ be the set of all $t$-preconvergences
$B$ of size $k$ that satisfy $\operatorname*{Cycs}B=Z$. Our above proof of
Proposition \ref{prop.new3} gives a bijection%
\[
\left\{  B\in U_{k,Z}\ \mid\ S\left(  B\right)  =P\right\}  \rightarrow
\left\{  B\in U_{k,Z}\ \mid\ T\left(  B\right)  =Q\right\}
\]
for every pair of nonempty subsets $P,Q\subseteq V$ satisfying $V=P\sqcup Q$
(essentially, this bijection sends each $B\in U_{k,Z}$ satisfying $S\left(
B\right)  =P$ to the set $\left(  B\setminus A\left(  P,Q\right)  \right)
\cup\rho_{P,Q}\left(  B\cap A\left(  P,Q\right)  \right)  $, where $\rho
_{P,Q}$ is a fixed size-preserving bijection from $A\left(  P,Q\right)  $ to
$A\left(  Q,P\right)  $). Combining these bijections for all such pairs
$\left(  P,Q\right)  $, we obtain a bijection $\phi:U_{k,Z}\setminus
\Gamma_{k,Z}\left(  s\right)  \rightarrow U_{k,Z}\setminus\Gamma_{k,Z}\left(
t\right)  $. Now we need to construct a bijection $\psi:\Gamma_{k,Z}\left(
s\right)  \rightarrow\Gamma_{k,Z}\left(  t\right)  $ from it. The involution
principle tells us that this $\psi$ acts on a given element $B\in\Gamma
_{k,Z}\left(  s\right)  $ by repeatedly applying $\phi^{-1}$ to it until it no
longer belongs to $U_{k,Z}\setminus\Gamma_{k,Z}\left(  t\right)  $ (which
means that it belongs to $\Gamma_{k,Z}\left(  t\right)  $). It is not hard to
see that this does not require more than $\left\vert V\right\vert $ many
iterations, since $\left\vert T\left(  B\right)  \right\vert $ increases with
each application of $\phi^{-1}$.

In the case of Proposition \ref{prop.A-symmetry}, however, we can also give a
direct bijective proof: By the directed Euler--Hierholzer theorem, each weak
component of the balanced digraph $D$ has an Eulerian circuit. Pick such a
circuit for each weak component of $D$. Note that on each of these circuits,
arcs from $A\left(  P,Q\right)  $ and arcs from $A\left(  Q,P\right)  $
alternate (if we remove the arcs from $A\left(  P,P\right)  $ and $A\left(
Q,Q\right)  $). Thus, we can define a bijection $A\left(  P,Q\right)
\rightarrow A\left(  Q,P\right)  $ that sends each arc $a\in A\left(
P,Q\right)  $ to the next arc $b\in A\left(  Q,P\right)  $ following it on the
chosen Eulerian circuit. At a second thought, this does not even require
Eulerian circuits; it suffices to pick any decomposition of $A$ into
(arc-disjoint) circuits.

\bigskip

\textbf{2.} We can also characterize $s$-convergences in terms of sinks.
Recall that a \emph{sink} of a digraph means a vertex with no outgoing arcs
(i.e., a vertex with outdegree $0$). Now our alternative characterization of
$s$-convergences is as follows:

\begin{proposition}
\label{prop.s-conv.2}Let $B$ be an acyclic subset of $A$. Let $s\in V$. Then,
$B$ is an $s$-convergence if and only if $s$ is the only sink of the
subdigraph $D\left\langle B\right\rangle $.
\end{proposition}

\begin{proof}
$\Longrightarrow$: Assume that $B$ is an $s$-convergence. Thus, $s$ is a
to-root of $D\left\langle B\right\rangle $.

We shall first show that $s$ is a sink of $D\left\langle B\right\rangle $.
Indeed, assume the contrary; thus, there exists an arc $a$ of $D\left\langle
B\right\rangle $ with source $s$. Let $v$ be the target of this arc $a$. Since
$s$ is a to-root of $D\left\langle B\right\rangle $, there exists a path from
$v$ to $s$ in $D\left\langle B\right\rangle $. This path, together with the
arc $a$, creates a cycle\footnote{If $v=s$, then this cycle consists of just a
single loop.} in $D\left\langle B\right\rangle $, which contradicts the
acyclicity of $B$. So, we see that $s$ is a sink of $D\left\langle
B\right\rangle $.

Furthermore, no vertex $v\neq s$ of $D\left\langle B\right\rangle $ can also
be a sink of $D\left\langle B\right\rangle $, for $D\left\langle
B\right\rangle $ must have a path from $v$ to $s$ (because $s$ is a to-root of
$D\left\langle B\right\rangle $) and this path must begin with an arc with
source $v$. Hence, $s$ is the only sink of $D\left\langle B\right\rangle $.
\medskip

$\Longleftarrow$: Assume that $s$ is the only sink of $D\left\langle
B\right\rangle $. Let $v\in V$ be any vertex.

The digraph $D\left\langle B\right\rangle $ has no cycles (since $B$ is
acyclic). Thus, any walk of $D\left\langle B\right\rangle $ is a path (since a
walk that is not a path would contain a cycle). Consequently, the walks of
$D\left\langle B\right\rangle $ are precisely the paths of $D\left\langle
B\right\rangle $. Hence, the digraph $D\left\langle B\right\rangle $ has only
finitely many walks that start at $v$ (since it clearly has only finitely many
\textbf{paths} that start at $v$). Moreover, there is at least one such walk
(the length-$0$ walk $\left(  v\right)  $).

Thus, there is a longest such walk. Let $\mathbf{p}$ be such a walk. Then,
$\mathbf{p}$ must end at a sink of $D\left\langle B\right\rangle $ (since
otherwise, we could extend $\mathbf{p}$ by an additional arc at the end,
obtaining an even longer walk that starts at $v$). Since the only sink of
$D\left\langle B\right\rangle $ is $s$, this means that $\mathbf{p}$ must end
at $s$. Thus, $\mathbf{p}$ is a walk from $v$ to $s$, hence a path from $v$ to
$s$ (since any walk of $D\left\langle B\right\rangle $ is a path). Hence,
$D\left\langle B\right\rangle $ has a path from $v$ to $s$ (namely,
$\mathbf{p}$). Since $v$ was arbitrary, this shows that $s$ is a to-root of
$D\left\langle B\right\rangle $. Therefore, $B$ is an $s$-convergence (since
$B$ is acyclic).
\end{proof}

\bigskip

\textbf{3.} In Lemma \ref{lem.new1}, Lemma \ref{lem.new2}, Lemma
\ref{lem.new4}, Lemma \ref{lem.new5} and Proposition \ref{prop.s-conv.2},
there is no need for the digraph $D$ to be balanced.

\bigskip

\textbf{4.} The relationships between the above theorems and some results in
the literature prompt some natural questions:

\begin{enumerate}
\item[\textbf{(a)}] Let $G=\left(  V,E,\varphi\right)  $ be an undirected
multigraph, and let $D=G^{\operatorname*{bidir}}$. Then, in \cite[Proposition
5.3]{NoPoSt02}, Novik, Postnikov and Sturmfels show that the number
$\delta_{\left\vert E\right\vert }\left(  s\right)  $ for any $s\in V$ is the
M\"{o}bius invariant of the intersection lattice of the graphic arrangement of
$G$. Are there some similar interpretations for $\delta_{k}\left(  s\right)  $
for other $k$, and perhaps for the general case when $D\neq
G^{\operatorname*{bidir}}$ ?

\item[\textbf{(b)}] Again, let $G=\left(  V,E,\varphi\right)  $ be an
undirected multigraph. In \cite[Theorem 3.4]{Gioan07}, Gioan shows that the
number of indegree functions of orientations of $G$ that have a given vertex
$s$ as a to-root (or \textquotedblleft unique quasi-sink\textquotedblright\ in
his lingo) is independent of $s$, and equals the evaluation $t\left(
G;1,1\right)  $ of the Tutte polynomial of $G$. This is not quite our
$\delta_{\left\vert E\right\vert }\left(  s\right)  $, since we are counting
the orientations themselves rather than their indegree functions; but it
suggests that the Tutte polynomial is not far away. The Tutte polynomial can be
generalized to digraphs \cite{AwaBer18}; does the relationship generalize?
\end{enumerate}

\subsection*{Acknowledgments}

We thank Karla Leipold for her NORCOM 2025 talk, which gave
the initial motivation to this work, and
the anonymous referees for helpful remarks.

\end{document}